\documentclass[11pt]{amsart}

\usepackage{epigamath}

\usepackage[english]{babel}

\numberwithin{equation}{section}

\usepackage{enumitem}

\newtheorem{theorem}{Theorem}[section]
\newtheorem{corollary}[theorem]{Corollary}

\newtheorem{proposition}[theorem]{Proposition}
\newtheorem{conjecture}[theorem]{Conjecture}

\theoremstyle{definition}

\newtheorem{definition}[theorem]{Definition}

\newtheorem{remark}[theorem]{Remark}

\EpigaVolumeYear{4}{2020} \EpigaArticleNr{18} \ReceivedOn{November 3,
2019}
\InFinalFormOn{September 3, 2020}
\AcceptedOn{October 5, 2020}

\title{Rationally connected rational double covers\\of primitive Fano varieties}
\titlemark{Rationally connected rational double covers of primitive Fano varieties}

\author{Aleksandr V. Pukhlikov}
\address{Department of Mathematical Sciences, The University of Liverpool}
\email{pukh@liverpool.ac.uk}

\authormark{A. Pukhlikov}

\AbstractInEnglish{We show that for a Zariski general
hypersurface $V$ of degree $M+1$ in ${\mathbb P}^{M+1}$ for
$M\geqslant 5$ there are no Galois rational covers
$X\dashrightarrow V$ of degree $d\geqslant 2$ with an abelian Galois group, where $X$ is a
rationally connected variety. In particular, there are no rational
maps $X\dashrightarrow V$ of degree 2 with $X$ rationally
connected. This fact is true for many other families of primitive
Fano varieties as well and motivates a conjecture on absolute
rigidity of primitive Fano varieties.}



\TitleInFrench{Rev\^etements rationnels doubles rationnellement connexes des vari\'et\'es de Fano primitives}

\AbstractInFrench{Nous montrons qu'une hypersurface $V$ de degr\'e $M+1$ de ${\mathbb P}^{M+1}$ avec $M\geqslant 5$ n'admet aucun rev\^etement rationnel galoisien $X\dashrightarrow V$de degr\'e $d\geqslant 2$ et de groupe de Galois ab\'elien avec $X$ rationnellement connexe. En particulier, $V$ n'admet aucune application rationnelle $X\dashrightarrow V$ de degr\'e 2 avec $X$ rationnellement connexe. Ceci est \'egalement v\'erifi\'e par beaucoup d'autres familles de vari\'et\'es de Fano primitives et motive une conjecture concernant la rigidit\'e absolue des vari\'et\'es de Fano primitives.}

\acknowledgement{The author thanks The Leverhulme Trust for the support (Research Project Grant RPG-2016-279).}

\begin{document}



\maketitle

\begin{prelims}

\DisplayAbstractInEnglish

\bigskip

\DisplayKeyWords

\medskip

\DisplayMSCclass

\bigskip

\languagesection{Fran\c{c}ais}

\bigskip

\DisplayTitleInFrench

\medskip

\DisplayAbstractInFrench

\end{prelims}


\newpage

\setcounter{tocdepth}{2}

\tableofcontents

\section{Statement of the main results}

One of the most challenging problems in the modern birational geometry is the \emph{unirationality problem}: for a given rationally connected
projective variety $V$, is there a rational dominant map ${\mathbb
P}^M\dashrightarrow V$? While in the past 50 years a huge progress
has been made in solving the rationality problem, the
unirationality still remains an unknown territory. There are lots
of explicit examples and constructions of unirational varieties
but not a single result about non-unirationality (in the
rationally connected category). During the past century (starting
with Fano himself and maybe even earlier) the non-unirationality
was conjectured for various classes of algebraic varieties, but
the questions remained unanswered.

The aim of the present paper is to prove a theorem that implies,
in particular, that there are no rational maps of degree 2
$$
{\mathbb P}^M\stackrel{2:1}{\dashrightarrow} V
$$
for a Zariski general hypersurface $V\subset {\mathbb P}^{M+1}$ of
degree $M+1$ (which is a Fano variety of index 1, in particular, a
rationally connected variety). In fact, there are no rational maps
of degree 2
$$
X\stackrel{2:1}{\dashrightarrow} V
$$
with $X$ rationally connected for such hypersurfaces $V$. We now
state the main theorem.

Let $V$ be a projective factorial variety with at most terminal
singularities, such that $\mathop{\rm Pic} V={\mathbb Z} K_V$ and
the anticanonical class $(-K_V)$ is ample (that is, a primitive
Fano variety). We work over the ground field ${\mathbb C}$ of
complex numbers.

\begin{definition}[\emph{cf.} \protect{\cite{Pukh05}}]\label{def1}
The Fano variety $V$, described
above, is \emph{divisorially canonical} if for every effective
divisor $D\sim -nK_V$, $n\geqslant 1$, the pair $(V,
\frac{1}{n}D)$ is canonical; that is to say, for every exceptional
prime divisor $E$ over $V$ the inequality
$$
\mathop{\rm ord\,}\nolimits_E D\leqslant n\cdot a(E),
$$
where $a(E)$ is the discrepancy of $E$ with respect to $V$, holds.
\end{definition}

Apart from the divisorial canonicity, we will need the following
technical conditions.
\begin{itemize}
\item[($\star$1)] For every anticanonical divisor $R\in |-K_V|$, every prime
number $p\geqslant 2$ and any, possibly reducible, closed subset
$Y\subset V$ of codimension $\geqslant 2$ there is a non-singular
curve $N\subset V$ such that
$$
p\not|\, (N\cdot K_V),
$$
$N\cap Y=\emptyset$ and $N$ meets $R$ transversally at
non-singular points.
\item[($\star$2)] For every, possibly reducible, closed subset $Y\subset V$ of
codimension $\geqslant 2$ there is a non-singular rational curve
$N\subset V$ such that $N\cap Y=\emptyset$.
\end{itemize}

\begin{definition}
We say that a rational dominant map
$X\dashrightarrow Z$ of varieties of the same dimension is a \emph{rational Galois cover}, if the corresponding field extension
${\mathbb C} (Z)\subset {\mathbb C}(X)$ is a Galois extension. If
the corresponding Galois group is cyclic, we say that this
rational map is a \emph{rational cyclic cover}.
\end{definition}

The main result of the present paper is the following claim.

\begin{theorem}\label{th1}
Assume that the Fano variety $V$, introduced
above, is divisorially canonical and satisfies the conditions $(\star1)$
and $(\star2)$. Then there are no rational Galois covers
$X\stackrel{d:1}{\dashrightarrow} V$ with an abelian Galois group
of order $d\geqslant 2$, where $X$ is a rationally connected
variety.
\end{theorem}

Since the divisorial canonicity implies birational superrigidity,
as an immediate consequence of Theorem~\ref{th1}, we get the following
claim.

\begin{corollary}\label{cor1}
 In the assumptions of Theorem~\ref{th1}, if
$X\stackrel{d:1}{\dashrightarrow} V$ is a rational Galois cover
with an abelian Galois group, then $d=1$ and the MMP for $X$ has
the unique outcome $V$.
\end{corollary}

Since every rational map of degree 2 is a Galois rational cover
with the cyclic group $C_2$ as the Galois group, we obtain the
following claim.

\begin{corollary}\label{cor2}
 In the assumptions of Theorem~\ref{th1}, there are
no rational maps $X\stackrel{2:1}{\dashrightarrow} V$ of degree 2
with $X$ a rationally connected variety.
\end{corollary}

Although a particular case of Theorem~\ref{th1}, the last corollary is
especially important as it covers \emph{all} rational maps of
degree 2 and therefore motivates the following conjecture.

\begin{conjecture}[on absolute rigidity]\label{conj1}
 If $V$ is a
divisorially canonical Fano variety, then every rational dominant
map $X\dashrightarrow V$, where $X$ is a rationally connected
variety of dimension $\mathop{\rm dim\,} V$, is a birational map.
\end{conjecture}

\section{Divisorially canonical varieties}

Given that the main assumption for the variety $V$ in Theorem~\ref{th1} is divisorial
canonicity, the natural question to ask now is how typical this
property is in the class of Fano varieties? Let the symbol
${\mathbb P}$ stand for the complex projective space ${\mathbb P}^{M+1}$, where $M\geqslant 5$. Set ${\mathcal F}={\mathbb
P}(H^0({\mathbb P},{\mathcal O}_{\mathbb P}(M+1)))$ to be the space of
hypersurfaces of degree $M+1$ in ${\mathbb P}$. If a hypersurface
$V\in {\mathcal F}$ is factorial, then $\mathop{\rm Pic} V={\mathbb
Z}H$, where $H$ is the class of a hyperplane section. In
\cite{Pukh05} it was shown that a Zariski general \emph{non-singular} hypersurface $V$ is divisorially canonical. Since it
is not hard to check that the properties ($\star$1) and ($\star$2) are
satisfied for a non-singular hypersurface $V$ (this is done below
in Section~\ref{section3}), we obtain the following claim.

\begin{corollary}\label{cor3}
 For a Zariski general hypersurface
$V\subset {\mathbb P}$ of degree $M+1$, where $M\geqslant 5$,
there are no non-trivial rational Galois covers
$X\stackrel{d:1}{\dashrightarrow} V$ with an abelian Galois group
of order $d\geqslant 2$, where $X$ is a rationally connected
variety; in particular, there are no rational maps
$X\dashrightarrow V$ of degree 2 with $X$ rationally connected
\end{corollary}

Note that in \cite{Pukh15a} it was shown that for $M\geqslant 9$
there exists a Zariski open subset ${\mathcal F}_{\rm reg}\subset
{\mathcal F}$, such that every hypersurface $V\in {\mathcal F}_{\rm reg}$
has at most quadratic singularities of rank $\geqslant 8$, so is a
factorial variety with terminal singularities, and satisfies the
property of divisorial canonicity. Moreover, for the complement
${\mathcal F}\setminus{\mathcal F}_{\rm reg}$ the inequality
$$
\mathop{\rm codim\,}(({\mathcal F}\setminus {\mathcal F}_{\rm reg})\subset
{\mathcal F})\geqslant \frac{(M-6)(M-5)}{2}-5
$$
holds. The property ($\star$1) is easy to show for hypersurfaces $V\in
{\mathcal F}_{\rm reg}$, see Section~\ref{section3}.

Apart from Fano hypersurfaces of index 1, the divisorial
canonicity was also shown for Zariski general varieties in the
following families:

\begin{itemize}

\item double spaces of index 1 and dimension $\geqslant 3$,
see \cite{Pukh05},

\item a majority of the families of Fano complete intersections
of index 1 in the projective space
\cite{Pukh06b,EcklPukh2016,Pukh18a},

\item finite, not necessarily cyclic, covers of index 1 of the
projective spaces \cite{Pukh19a}.

\end{itemize}

This list is probably not complete: computing or estimating the
(log) canonical thresholds has become a popular topic, see
\cite{ChParkWon2014,ChShr2008,LiuZhuang2018,Zhuang2018} and other
works in this direction.

\section{Proof of Theorem~\ref{th1}}\label{section3}

Let us assume the converse and fix a
non-trivial rational Galois cover $\sigma\colon X\dashrightarrow
V$ with an abelian Galois group, where $V$ is a divisorially
canonical variety, satisfying ($\star$1) and ($\star$2), and $X$ is rationally
connected. Considering the field extension ${\mathbb C}(V)\subset
{\mathbb C}(X)$, we can find an intermediate field which is a
normal extension of ${\mathbb C}(V)$ with a cyclic group of a
prime order $p\geqslant 2$ as its Galois group. Since the image of
a rationally connected variety is rationally connected, we may
assume that the Galois group of the original extension ${\mathbb
C}(V)\subset {\mathbb C}(X)$ is a cyclic group of a prime order
$p\geqslant 2$. Further, we may assume that $\sigma\colon X\to V$
is a morphism and $X$ is a non-singular projective variety,
$\mathop{\rm dim\,} X = \mathop{\rm dim\,} V$.

We say that a family ${\mathcal L}$ of irreducible projective curves
on a quasi-projective variety is {\it free}, if they sweep out a
dense subset of that variety and for every subvariety $Y$ of
codimension $\geqslant 2$ the subset
$$
\{ L\in {\mathcal L}\,|\, L\cap Y\neq \emptyset\}
$$
is a proper closed subfamily of the family ${\mathcal L}$ (that is to
say, a curve $L\in {\mathcal L}$ of general position does not
intersect $Y$). Let us fix a free family ${\mathcal C}_X$ of
non-singular rational curves on $X$.

This free family of rational curves is a crucial object in our
proof. For the existence and basic properties of free families of
rational curves, our reference is Koll\'{a}r's book, \cite[Sections~II.3 and~IV.3]{Kol96}. The facts used below are well known
and standard; here we explain briefly the easiest way to
understand the basic geometry of such families. By Theorem~3.9 in
\cite[Section~IV.3]{Kol96} we have a family ${\mathcal C}_X$ of
non-singular rational curves $C_X$ on $X$ such that the vector
bundle $T_X|_{C_X}$ is ample, that is, it is of the form
$$
\oplus {\mathcal O}_{{\mathbb P}^1}(\alpha_i)
$$
with all $\alpha_i\geqslant 1$. (This implies, in particular, that
$(K_X\cdot C_X)<0$, although that inequality holds under much
weaker assumptions for the family of curves.) Therefore, the
infinitesimal deformations of the curve $C_X$ given by the
sections of the vector bundle $T_X|_{C_X}$ are unobstructed, see
Points 3.3 -- 3.5.4 in \cite[Section~II.3]{Kol96}, and if for
a particular subvariety $Y\subset X$ of codimension $\geqslant 2$
and a particular curve $C_X\in {\mathcal C}_X$ the intersection
$C_X\cap Y\neq \emptyset$, then a general deformation of $C_X$ in
the family ${\mathcal C}_X$ does not meet $Y$: we can deform the curve
away from $Y$, see Proposition~3.7 in \cite[Section~II.3]{Kol96}.

For the same reason, for every prime divisor $\Delta\subset X$,
such that $\sigma_*\colon T_pX\to T_{\sigma(p)}V$ is not an
isomorphism for a point of general position $p\in \Delta$ (this is
true, in particular, if $\mathop{\rm
codim\,}(\sigma(\Delta)\subset V)\geqslant 2$), a general curve
$C_X\in {\mathcal C}_X$ meets $\Delta$ transversally at points of
general position. (There are finitely many such divisors, so the
assumption that a general $C_X$ meets every $\Delta$ at points of
general position is justified by the family ${\mathcal C}_X$ being
free; the transversality follows from the ampleness of the
restriction of the tangent bundle $T_X$ onto $C_X$, see the
references above.)

The same deformation arguments (see the proof of Proposition~3.7
in \cite[Section~II.3]{Kol96} and Theorems 1.7 and 1.8 in
\cite[Section II.1]{Kol96} give us that for a general curve
$C_X\in {\mathcal C}_X$ the morphism
$$
\sigma|_{C_X}\colon C_X\to \sigma(C_X)
$$
is birational (again, this comes from the ampleness of $T_X$
restricted onto a general $C_X$, see Proposition 3.5 and Corollary
3.5.3 in \cite[Section II.3]{Kol96}: when we deform the curve
$C_X$, any two distinct points $p\neq q$ on this curve vary
independently, so that if for a particular curve $C_X$ the
morphism $\sigma|_{C_X}$ is not birational, choosing two points
$p\neq q$ such that $\sigma(p)=\sigma(q)$, we can see that a
general deformation of this curve satisfies the required
property.)

Since in the subsequent arguments we need only the general curve
$C_X\in {\mathcal C}_X$, we will remove from ${\mathcal C}_X$ proper
closed subsets when we need it, without special comments and
keeping the same notation ${\mathcal C}_X$. Let ${\mathcal C}_V=\sigma_*
{\mathcal C}_X$ be the image of that family on $V$. The family ${\mathcal
C}_V$ is, generally speaking, not free: if the $\sigma$-image of a
prime divisor $\Delta\subset X$ is of codimension $\geqslant 2$,
then the general curve $C_V\in {\mathcal C}_V$ meets $\sigma(\Delta)$.

\begin{proposition}\label{prop1}
There is a birational morphism
$\varphi\colon V^+\to V$, where $V^+$ is a non-singular projective
variety, such that the strict transform ${\mathcal C}_V^+$ of the
family ${\mathcal C}_V$ on $V^+$ is a free family of curves.
\end{proposition}

\begin{proof}
It is given in Section~\ref{section4}. (Here of course the strict
transform of a family of irreducible curves is the family of
irreducible curves, a general curve in which is the strict
transform of a general curve in the original family. The
parameterizing space of the new family is, generally speaking, a
Zariski open subset of the parameterizing space of the original
family.)
\end{proof}

\begin{proposition}\label{prop2}
There is a non-singular quasi-projective
variety $U_X$, a birational map $\varphi_X\colon
U_X\dashrightarrow X$ and a Zariski open subset $U\subset V^+$,
such that:
\begin{itemize}
\item[{\rm (i)}] the rational map
$$
\sigma_*=\varphi^{-1}\circ\sigma\circ\varphi_X\colon
U_X\dashrightarrow V^+
$$
extends to a morphism $\sigma_U\colon U_X\to V^+$, the image of
which is $U$,
\item[{\rm (ii)}]  the inequality
$$
\mathop{\rm codim\,}((V^+\setminus U)\subset V^+)\geqslant 2
$$
holds,
\item[{\rm (iii)}]  the map $\sigma_U\colon U_X\to U$ is a cyclic cover of
order $p$, branched over a non-singular hypersurface $W\subset
U$.
\end{itemize}
\end{proposition}
\noindent The proof is given in Section~\ref{section5}.

\vskip\baselineskip
To make the statement of Proposition~\ref{prop2} more visual, we arrange the
maps into the following commutative diagram:
$$
\begin{array}{rcccccl}
   &   U_X &    & \stackrel{\varphi_X}{\dashrightarrow} &   & X &
\\
\sigma_U\!\!\! & \downarrow &  &  &  &  \downarrow & \!\!\! \sigma
\\
   &   U   & \subset & V^+ & \stackrel{\varphi}{\to} & V. &
\end{array}
$$

Assuming Propositions~\ref{prop1} and~\ref{prop2}, let us complete the proof of
Theorem~\ref{th1}. By the inequality (ii) of Proposition~\ref{prop2} a general curve
$C\in {\mathcal C}^+_V$ does not meet the closed set $V^+\setminus U$
and for that reason is contained entirely in $U$. Therefore, for
some open subfamily ${\mathcal C}\subset {\mathcal C}^+_V$ all curves
$C\in {\mathcal C}$ are entirely contained in $U$, so that ${\mathcal C}$
is a family of irreducible projective rational curves sweeping out
$U$. By the construction of the family ${\mathcal C}$ for a general
curve $C\in {\mathcal C}$ its preimage
$$
\sigma^{-1}_U(C)=C_1\cup C_2\cup \cdots \cup C_p
$$
is a union of $p$ distinct rational curves on the quasi-projective
variety $U_X$. Indeed, since for a general curve $C_X\in {\mathcal
C}_X$ the morphism $\sigma|_{C_X}$ is birational, and by
construction of the family ${\mathcal C}$, we may assume that $C_1$
belongs to a family of irreducible rational curves sweeping out
$U_X$ and the morphism $C_1\to C$ is birational. Then the other
curves $C_i$, $i\neq 1$, are the images of $C_1$ under the action
of elements of the cyclic Galois group; in particular, they also
belong to families of irreducible rational curves sweeping out
$U_X$, so that
$$
(C_{i}\cdot K_U)<0,
$$
$i=1,\dots, p$, where $K_U$ is the canonical class of the variety
$U_X$. Write
$$
K^+_V=-\varphi^* H+\sum_{i\in I}a_i E_i
$$
for the canonical class of the variety $V^+$, where $E_i\subset
V^+$ are all the prime $\varphi$-exceptional divisors and $a_i>0$
are their discrepancies with respect to $V$ and $H=-K_V$ is the
ample anticanonical generator of $\mathop{\rm Pic} V$. (The
notation $K^+_V$ looks better than the standard symbol $K_{V^+}$.)
Denoting the restrictions of the divisorial classes onto $U$ by
the same symbols and omitting the symbol $\varphi^*$, we get
\begin{equation}\label{14.10.2019.1}
K_U=\sigma^*_U\left(K^+_V+\left(1-\frac{1}{p}\right)W\right),
\end{equation}
where $W\subset U$ is a non-singular hypersurface, over which the
cyclic cover $\sigma_U$ is branched. Collecting separately the
components of the hypersurface $W$, which are divisorial on $V$
and $\varphi$-exceptional, write
$$
W=W_{\rm div}+W_{\rm exc},
$$
where $W_{\rm div}=nH-\sum_{i\in I}b_i E_i$ with $b_i\in {\mathbb
Z}_+$ and $W_{\rm exc}=\sum_{i\in I}c_i E_i$ with $c_i\in\{0,1\}$.
Obviously, the inequalities $(C\cdot K^+_V)<0$, $(C\cdot W_{\rm
div})\geqslant 0$ and $(C\cdot W_{\rm exc})\geqslant 0$ hold. At
the same time,
\begin{equation}\label{14.10.2019.2}
(C_{i}\cdot K_U)=(C\cdot K^+_V)+\left(1-\frac{1}{p}\right)(C\cdot
W)<0.
\end{equation}
Assume first that $n\geqslant 2$. Adding to the left hand side of
(\ref{14.10.2019.2}) the non-positive expression
$[\left(1-\frac{1}{p}\right)n-1](C\cdot K^+_V)$, we obtain the
inequality
\begin{equation}\label{14.10.2019.3}
n(C\cdot K^+_V)+(C\cdot W)=\left(C\cdot\sum_{i\in I}(na_i-b_i+c_i)
E_i\right)<0,
\end{equation}
so that for some $i\in I$ we have
$$
b_i>na_i+c_i\geqslant n\cdot a_i
$$
and the pair $(V,\frac{1}{n}\varphi_* W_{\rm div})$ with
$\varphi_* W_{\rm div}\sim nH$ is not canonical, which contradicts
the divisorial canonicity of the variety $V$. Therefore, $n=1$ or
$0$.

The case $n=1$ is impossible: in that case $\varphi_* W_{\rm
div}=\varphi(W_{\rm div})\subset V$ is an anticanonical divisor on
the variety $V$ and we apply the assumption ($\star$1) for the prime
divisor $R=\varphi(W_{\rm div})$: we take a non-singular curve $N$
(of arbitrary genus) such that $p\not| (N\cdot K_V)$, which does
not meet the set
\begin{equation}\label{09.10.2019.1}
\varphi(V^+\setminus U)\cup
\varphi\left(\mathop{\bigcup}\limits_{i\in I}E_i\right)
\end{equation}
and meets the divisor $R$ transversally at points of general
position. Its strict transform $N^+$ is contained entirely in $U$,
and moreover,
$$
\sigma^{-1}_U(N^+)\to N^+
$$
is a cyclic cover of a non-singular curve, branched over a set of
$(N\cdot R)$ points, the number of which is not divisible by $p$,
which is impossible.

The case $n=0$ is also impossible. Here we apply the assumption
($\star$2), taking on the variety $V$ a non-singular rational curve $N$,
which does not meet the set (\ref{09.10.2019.1}). Then the strict
transform $N^+\subset V^+$ is contained entirely in $U$, so that
$$
\sigma^{-1}_U(N^+)\to N^+
$$
is a non-ramified $p$-cyclic cover of a non-singular rational
curve, which is impossible. \qed

\begin{proof}[Proof of Corollary~\ref{cor3}]
Given that Zariski general
hypersurfaces $V\in {\mathcal F}$ were shown in \cite{Pukh05} to be
divisorially canonical, it remains to check that ($\star$1) and ($\star$2)
hold for a general hypersurface $V$. In fact, they are satisfied
for any non-singular hypersurface $V\subset {\mathbb P}$ of degree
$(M+1)$: ($\star$2) is true because $V$ is a non-singular Fano variety,
and for ($\star$1) we have the following simple argument. If
$p\mathbin{\not|} \mathop{\rm deg} V$, then a section of $V$ by a
general 2-plane in ${\mathbb P}$ does the job. Let us assume that
$p\mathbin{|} \mathop{\rm deg} V$.

A general line $L\subset V$ meets the hyperplane section $R$
transversally at one point. However, it might happen that $L\cap
Y\neq \emptyset$. In any case, since the lines on $V$ sweep out a
divisor, we may assume that $L\not\subset Y$,so that $L\cap Y$ is
a finite set of points.

Now take a general plane $P\subset {\mathbb P}$ containing $L$,
and let $N\subset P$ be the residual curve of the intersection
$P\cap V=L+N$. Making sure that $P$ is not contained in the
hyperplanes in ${\mathbb P}$ that are tangent to $V$ at the points
in $L\cap Y$ and taking into account that the image of $Y$ with
respect to the projection from the line $L$ is a proper closed
subset of ${\mathbb P}^{M-1}$, we may assume that $N\cap
Y=\emptyset$. As
$$
\mathop{\rm deg} N = \mathop{\rm deg} V -1
$$
and $p\mathbin{|} \mathop{\rm deg} V$, the curve $N$ is what we
need.
\end{proof}

\begin{remark}\label{rk1}
The property ($\star$1) is easy to show for singular
hypersurfaces $V\in {\mathcal F}_{\rm reg}$, considered in
\cite{Pukh15a}. Recall \cite[Section 3, Subsection 3.2]{Pukh15a}
that the conditions defining the open set ${\mathcal F}_{\rm reg}$,
include the following bound for the singularities of $V$: the
hypersurface may have at most quadratic singularities of rank
$\geqslant 8$, so that, in particular,
$$
\mathop{\rm codim\,} (\mathop{\rm Sing} V\subset V)\geqslant 7.
$$
Therefore, a general 2-plane in ${\mathbb P}$ does not meet the
closed set $\mathop{\rm Sing} V$. Furthermore, the regularity
conditions at every singular point $o\in \mathop{\rm Sing} V$
ensure that there are finitely many lines through the point $o$ on
$V$. Therefore, a general line $L\subset V$ does not meet the
singular locus of $V$ and the proof of ($\star$1), given above for a
non-singular hypersurface $V$, works for any $V\in {\mathcal F}_{\rm
reg}$ word for word.
\end{remark}

\section{Resolution of a family of curves}\label{section4}

In order to prove Proposition~\ref{prop1}, let us construct a sequence of birational morphisms
\begin{equation}\label{08.10.2019.1}
V_0\stackrel{\varphi_1}{\longleftarrow}
V_1\stackrel{\varphi_2}{\longleftarrow}\cdots
\stackrel{\varphi_{M-1}}{\longleftarrow} V_{M-1}
\end{equation}
of non-singular projective varieties, such that $V_0=V$, ${\mathcal
C}_0={\mathcal C}_V$ and for every closed subset $Y\subset V_i$ of
dimension $\leqslant i-1$ a general curve $C_i\in {\mathcal C}_i$ of
the strict transform of the family ${\mathcal C}_V$ on $V_i$ does not
meet the subset $Y$.

For $i=0$ the last claim holds in a trivial way. We will explain
in detail the first two steps of this construction: the morphisms
$\varphi_1$ and $\varphi_2$. The general step $\varphi_j\colon
V_j\to V_{j-1}$ is very similar to $\varphi_2$ and will be easy to
understand when $\varphi_2$ is clear.

{\bf Step 1.} Let us consider the family ${\mathcal C}_0={\mathcal C}_V$
of curves, sweeping out the variety $V_0$. It is clear that the
set of points $\Xi_0\subset V_0$, which are contained in all
curves of the family ${\mathcal C}_0$, is finite. We blow up this
finite set of points:
$$
\varphi_{1,1}\colon V_{0,1}\to V_{0,0}=V_0,
$$
and look at the strict transform ${\mathcal C}_{0,1}$ of the family
${\mathcal C}_{0}$ on $V_{0,1}$. If there are no points that lie on
all curves in ${\mathcal C}_{0,1}$, we stop. Otherwise, there is a
finite set $\Xi_{0,1}\subset V_{0,1}$ of such points (obviously,
$\Xi_{0,1}$ is contained in the exceptional divisor of
$\varphi_{1,1}$) and we blow it up. Repeating, if necessary, we
get a finite sequence of blow ups of finite sets of points,
$$
V_0=V_{0,0}\stackrel{\varphi_{1,1}}{\longleftarrow}
V_{0,1}\stackrel{\varphi_{1,2}}{\longleftarrow}\cdots
\stackrel{\varphi_{1,e(1)}}{\longleftarrow} V_{0,e(1)},
$$
such that for every point $p\in V_{0,e(1)}$ the curves
$C_{0,e(1)}\in {\mathcal C}_{0,e(1)}$ (the last symbol means the
strict transform of the family ${\mathcal C}_0$ on $V_{0,e(1)}$),
containing the point $p$, form a proper closed subset of the
family ${\mathcal C}_{0,e(1)}$. (That our procedure can not be
infinite and must terminate, follows from considering some formal
parameterizations of all branches of a general curve $C_0\in {\mathcal
C}_0$ at the points of the set $\Xi_0$.) Set $V_1=V_{0,e(1)}$ and
$$
\varphi_1=\varphi_{1,1}\circ\cdots\circ \varphi_{1,e(1)}\colon
V_1\to V_0.
$$

{\bf Step 2.} For the next step of our construction, consider the
set of curves $\Xi_1\subset V_1$, which intersect all curves of
the family ${\mathcal C}_1$. This is a finite set of curves. Indeed,
embedding the quasi-projective variety, parameterizing the curves
of the family ${\mathcal C}_1$, in some projective space, and
intersecting with the appropriate number of general hyperplanes,
we obtain a sufficiently mobile family of curves in that
quasi-projective variety, which gives us a family of surfaces
${\mathcal S}_1$ on $V_1$, such every surface $S_1\in {\mathcal S}_1$
contains all curves of the set $\Xi_1$. Now it is clear that the
set $\Xi_1$ is finite.

If $D_1$ is a general divisor of a very ample system on $V_1$,
then, intersecting the surfaces $S_1\in {\mathcal S}_1$ with $D_1$, we
obtain a family of curves on $D_1$, each of which contains the
points of the set $\Xi_1\cap D_1$. Now, arguing as at the previous
step, for the family of curves $(S_1\cap D_1, S_1\in {\mathcal S}_1)$,
we see that there is a finite sequence of birational morphisms
$$
V_1=V_{1,0}\stackrel{\varphi_{2,1}}{\longleftarrow}
V_{1,1}\stackrel{\varphi_{2,2}}{\longleftarrow}\cdots
\stackrel{\varphi_{2,e(2)}}{\longleftarrow} V_{1,e(2)},
$$
where $\varphi_{2,1}$ is a composition of two birational maps:

\begin{itemize}

\item a desingularization of the (reducible) 1-dimensional subset
$\Xi_1$, that is, a composition of finitely many blow ups of
points, transforming $\Xi_1$ into a disjoint union of non-singular
curves, and

\item the blow up of the non-singular strict transform of $\Xi_1$,

\end{itemize}

\noindent and, similarly, $\varphi_{2,i}$ first resolves the
singularities of the 1-dimensional set $\Xi_{1,i-1}\subset
V_{1,i-1}$, which is the union of all irreducible curves on
$V_{1,i-1}$, intersecting all curves of the family ${\mathcal
C}_{1,i-1}$ (the strict transform of the family ${\mathcal C}_1$ on
$V_{1,i-1}$), and then blows up the non-singular strict transform
of the reducible curve $\Xi_{1,i-1}$. Note that by construction
the image of every curve in $\Xi_{1,i-1}$ on $V_1$ is one of the
curves of the set $\Xi_1$ (otherwise, all curves of the family
${\mathcal C}_1$ would have passed through some point, which is not
true). Our procedure terminates for the same reason as at Step 1
(looking at the family of curves $(S_1\cap D_1, S_1\in {\mathcal
S}_1)$ that was used above). Finally, on $V_{1,e(2)}$ there are no
curves meeting all curves of the family ${\mathcal C}_{1,e(2)}$. We
set $V_2=V_{1,e(2)}$, ${\mathcal C}_2={\mathcal C}_{1,e(2)}$ and
$$
\varphi_2=\varphi_{2,1}\circ\cdots\circ \varphi_{2,e(2)}\colon
V_2\to V_1.
$$
Carrying on in the same spirit, we construct the whole sequence
(\ref{08.10.2019.1}). Let us consider briefly the general step of
our construction.

{\bf Step $j$.} By the previous arguments, if an irreducible
subvariety intersects all curves of the family ${\mathcal C}_{j-1}$ on
$V_{j-1}$, it has dimension $\geqslant j-1$. If there are no
irreducible subvarieties of dimension $j-1$, intersecting all
curves in ${\mathcal C}_{j-1}$, then there is nothing to do: we set
$V_j=V_{j-1}$ and $\varphi_j$ is the identity morphism. Otherwise,
let $\Xi_{j-1}\subset V_{j-1}$ be the union of all irreducible
subvarieties of dimension $j-1$ meeting all curves in ${\mathcal
C}_{j-1}$. Embedding the quasi-projective variety, parameterizing
the curves of the family ${\mathcal C}_{j-1}$ in a projective space
and intersecting it with the appropriate number of general
hyperplanes, we obtain a sufficiently mobile family of
subvarieties of dimension $j-1$ in that quasi-projective variety,
which gives us a family ${\mathcal S}_{j-1}$ of irreducible
subvarieties of dimension $j$ on $V_{j-1}$, such that every
subvariety $S_{j-1}\in{\mathcal S}_{j-1}$ contains $\Xi_{j-1}$; in
particular, $\Xi_{j-1}$ is a Zariski closed subset, and hence it
contains finitely many irreducible components. If
$$
D_1,\quad\dots,\quad D_{j-1}
$$
is a general tuple of divisors in a very ample linear system on
$V_{j-1}$, then, intersecting the subvarieties $S_{j-1}\in{\mathcal
S}_{j-1}$ with $D_1\cap\dots\cap D_{j-1}$, we obtain a family of
curves on $D_1\cap\dots\cap D_{j-1}$, each of which contains the
points of the set $\Xi_{j-1}\cap D_1\cap\dots\cap D_{j-1}$. Now,
arguing as at Step 2, we see that there is a finite sequence of
birational morphisms
$$
V_{j-1}=V_{j-1,0}\stackrel{\varphi_{j,1}}{\longleftarrow}
V_{j-1,1}\stackrel{\varphi_{j,2}}{\longleftarrow}\cdots
\stackrel{\varphi_{j,e(j)}}{\longleftarrow} V_{j-1,e(j)},
$$
where $\varphi_{j,1}$ is a composition of two birational maps:

\begin{itemize}

\item a desingularization of the (reducible) $(j-1)$-dimensional
closed subset $\Xi_{j-1}$, that is, a composition of finitely many
blow ups of subvarieties of dimension $\leqslant j-2$,
transforming $\Xi_{j-1}$ into a disjoint union of non-singular
subvarieties of dimension $j-1$, and

\item the blow up of the non-singular strict transform of
$\Xi_{j-1}$.

\end{itemize}

\noindent Similarly, $\varphi_{j,i}$ first resolves the
singularities of the closed set $\Xi_{j-1,i-1}\subset
V_{j-1,i-1}$, which is the union of all irreducible subvarieties
on $V_{j-1,i-1}$, intersecting all curves of the family ${\mathcal
C}_{j-1,i-1}$ (the strict transform of the family ${\mathcal C}_{j-1}$
on $V_{j-1,i-1}$), and then blows up the non-singular strict
transform of the closed set $\Xi_{j-1,i-1}$. Again, as at Step 2,
the image of every component of the closed set $\Xi_{j-1,i-1}$ on
$V_{j-1}$ is one of the components of the set $\Xi_{j-1}$ ---
otherwise, every curve in ${\mathcal C}_{j-1}$ would have met some
irreducible subvariety of dimension $\leqslant j-2$, which is not
true by the previous steps of our construction. Our procedure
terminates for the same reason as at Steps 1 and 2 (using the
family of curves
$$
(S_{j-1}\cap D_1\cap\dots\cap D_{j-1},S_{j-1}\in {\mathcal S}_{j-1})
$$
on $D_1\cap\dots\cap D_{j-1}$). Finally, on $V_{j-1, e(j)}$ there
are no subvarieties of dimension $\leq j-1$, meeting all curves of
the family ${\mathcal C}_{j-1, e(j)}$. We set $V_j=V_{j-1, e(j)}$,
${\mathcal C}_j={\mathcal C}_{j-1, e(j)}$ and
$$
\varphi_j=\varphi_{j,1}\circ\cdots\circ \varphi_{j,e(j)}\colon
V_j\to V_{j-1}.
$$

Having constructed the sequence (\ref{08.10.2019.1}), set
$V^+=V_{M-1}$ and
$$
\varphi=\varphi_0\circ\varphi_1\circ\cdots\circ\varphi_{M-1}\colon
V^+\to V.
$$
Obviously, $V^+$ has the property, described in the statement of
the proposition. \qed

\section{Construction of the cyclic cover}\label{section5}

Let us show Proposition~\ref{prop2}. The field extension ${\mathbb C} (V)\subset {\mathbb C}(X)$ is
generated by some element $\xi\in {\mathbb C}(X)$, satisfying the
equation
$$
\xi^p-q=0
$$
for some rational function $q\in {\mathbb C} (V)={\mathbb C}
(V^+)$. For some effective divisor $R$ on $V^+$ and sections
$a_0,a_1\in {\mathcal O}_{V^+}(R)$ we can write $q=a_0/a_1$. Consider
the hypersurface
$$
\{G=0\}\subset V^+\times {\mathbb P}^1_{(x_0:x_1)},
$$
where
$$
G=a_1x_1^p-a_0x_0^p\in H^0(V^+\times {\mathbb P}^1,\mathop{\rm
pr}\nolimits^*_V {\mathcal O}_{V^+}(R)\otimes \mathop{\rm
pr}\nolimits^*_{\mathbb P} {\mathcal O}_{{\mathbb P}^1}(p)),
$$
the symbols $\mathop{\rm pr}\nolimits_V$ and $\mathop{\rm
pr}\nolimits_{\mathbb P}$ mean the projections of $V^+\times
{\mathbb P}^1$ onto the first and second factor, respectively.
This hypersurface is, generally speaking, reducible, but has a
unique irreducible component $X_0$, such that $\mathop{\rm
pr}\nolimits_V (X_0)=V^+$.\vspace{0.1cm}

If for some prime divisor $T\subset V^+$ the hypersurface
$\{G=0\}$ contains $\mathop{\rm pr}\nolimits_V^{-1}(T)$, then
$a_0,a_1|_T\equiv 0$, so that, replacing $R$ by $R-mT$, where
$$
m=\min \{\mathop{\rm ord}\nolimits_T(a_0), \mathop{\rm
ord}\nolimits_T(a_1)\},
$$
and $a_0,a_1\in {\mathcal O}_{V^+}(R)$ by
$$
\frac{a_0}{s^m_T}, \frac{a_1}{s^m_T}\in {\mathcal O}_{V^+}(R-mT),
$$
where $s_T\in {\mathcal O}_{V^+}(T)$ is the section, corresponding to
the divisor $T$, we remove the component $\mathop{\rm
pr}\nolimits_V^{-1}(T)$. Therefore, we may assume from the
beginning that the sections $a_0,a_1\in {\mathcal O}_{V^+}(R)$ do not
vanish simultaneously on any prime divisor $T\subset V^+$ and
$\{G=0\}=X_0\subset V^+\times {\mathbb P}^1$ is an irreducible
hypersurface, by construction birational to the original variety
$X$.\vspace{0.1cm}

Now let us consider the singularities of the variety $X_0$. Assume
that for some prime divisor $T\subset V^+$ there is a subvariety
$T_X\subset \mathop{\rm Sing} X_0$, such that
$$
\mathop{\rm pr}\nolimits_V (T_X)=T.
$$
Let ${\mathcal T}$ be the set of all prime divisors on $V^+$ with that
property. By what was said above, we may assume that say,
$a_1|_T\not\equiv 0$, which implies that
$$
\mathop{\rm ord}\nolimits_T a_0\geqslant 2.
$$
Removing the set of common zeros of the sections $a_0,a_1\in {\mathcal
O}_{V^+}(R)$, the pairwise intersections of the divisors $T\in
{\mathcal T}$ (if $\sharp {\mathcal T}\geqslant 2$) and the sets of
singular points $\mathop{\rm Sing} T$ for all $T\in {\mathcal T}$, we
obtain a Zariski open set $U\subset V^+$, such that

\begin{itemize}

\item $\mathop{\rm codim\,} ((V^+\setminus U)\subset V^+)\geqslant 2$,

\item the morphism $X_0\cap \mathop{\rm pr}\nolimits_V^{-1} (U)\to U$ is
a finite morphism of degree p and

\item for every divisor $T\in {\mathcal T}$ the quasi-projective varieties
$T\cap U$ and $T_X\cap \mathop{\rm pr}\nolimits_V^{-1} (U)$ are
non-singular, and moreover, the projection $\mathop{\rm
pr}\nolimits_V$ gives an isomorphism of these varieties, which we
for simplicity of notations write down again as $T$ and $T_X$.

\end{itemize}

Now distinct varieties in ${\mathcal T}$ are disjoint.

Furthermore, we may assume that for every $T\in {\mathcal T}$ the
section $a_i$, $i=0,1$, which does not vanish identically on $T$,
is everywhere non-zero on $T\cap U$, so that if say
$a_1|_T\not\equiv 0$, then the hypersurface $X_0\cap \mathop{\rm
pr}\nolimits_V^{-1} (U)$ over a neighborhood of the divisor $T\cap
U$ is contained in the open subset
$$
\{x_0\neq 0\}=U\times {\mathbb A}^1_z,
$$
with $z=x_1/x_0$ and given by the equation
$$
a_1z^p-a_0=0,
$$
so that the corresponding subvariety $T_X\subset \mathop{\rm Sing}
X_0$ over $T\cap U$ is given in $T\times {\mathbb A}^1_z$ by the
equation $z=0$.

Thus we have constructed a locally trivial ${\mathbb P}^1$-bundle
${\mathcal X}_1$ over a non-singular quasi-projective variety $U$ (of
course, ${\mathcal X}_1=U\times {\mathbb P}^1$) and a hypersurface
$X_1=X_0\cap \mathop{\rm pr}\nolimits_V^{-1} (U)\subset {\mathcal
X}_1$ which is a cyclic cover of the prime order $p$ over $U$. The
projection ${\mathcal X}_1\to U$ will be denoted by the symbol
$\pi_1$. The cyclic cover $X_1\to U$ has all properties required
in Proposition~\ref{prop2} except for $X_1$ being non-singular. The rest of
our proof of the proposition is removing the singularities of
$X_1$. This has to be done carefully, preserving the cyclic cover.

We desingularize the covering variety in two steps. The first step
is constructing a sequence of locally trivial ${\mathbb
P}^1$-bundles over $U$:
$$
{\mathcal X}_1\stackrel{\beta_1}{\longleftarrow} {\mathcal
X}_2\stackrel{\beta_2}{\longleftarrow}\cdots
\stackrel{\beta_{k-1}}{\longleftarrow} {\mathcal X}_{k},
$$
where $\beta_i$ is an elementary birational transformation over
$U$,
$$
\begin{array}{rcccl}
   & {\mathcal X}_{i+1} & \stackrel{\beta_i}{\longrightarrow} & {\mathcal
   X}_i  &   \\
\pi_{i+1}\!\!\!\! & \downarrow &   & \downarrow & \!\!\!\! \pi_i \\
   & U & = & U , &
\end{array}
$$
defined in the following way. Assume that the strict transform
$X_i\subset {\mathcal X}_i$ of the hypersurface $X_1$ is singular
along a subvariety $\pi^{-1}_i (T)\cap X_i=T_i$ for some $T\in
{\mathcal T}$. With respect to a certain trivialization of the
${\mathbb P}^1$-bundle ${\mathcal X}_i/U$ over an open subset,
intersecting the divisor $T$, the hypersurface $X_i$ is given by
the equation
$$
a_{i,1}x_1^p-a_{i,0}x_0^p=0,
$$
where, say, $a_{i,1}|_T\not\equiv 0$ and $\mathop{\rm
ord}\nolimits_T a_{i,0}\geqslant 2$. Assume, furthermore, that in
fact the inequality
$$
\mathop{\rm ord}\nolimits_T a_{i,0}\geqslant p
$$
holds. Then the birational transformation $\beta_i\colon {\mathcal
X}_{i+1}\to {\mathcal X}_{i}$ is the composition of the blow up of the
subvariety $T_i$ (by construction, this non-singular subvariety is
a section of the ${\mathbb P}^1$-bundle $\pi_i^{-1}(T)\to T$) and
subsequent contraction of the strict transform of the hypersurface
$\pi_i^{-1}(T)$. Elementary computa\-tions with local parameters,
using the explicit presentation described above, show that locally
in a neighborhood of the generic point of the divisor $T$ the
hypersurface $X_{i+1}$ is given by the equation
$$
a_{i+1,1}x_1^p-a_{i+1,0}x_0^p=0,
$$
where $a_{i+1,1}|_T\not\equiv 0$ and $\mathop{\rm ord}\nolimits_T
a_{i+1,0}=\mathop{\rm ord}\nolimits_T a_{i,0} - p$. Therefore,
after finitely many elementary transformations of the ambient
${\mathbb P}^1$-bundle we obtain a locally trivial ${\mathbb
P}^1$-bundle
$$
\pi_k\colon {\mathcal X}_k\to U,
$$
such that if the strict transform $X_k\subset {\mathcal X}_k$ of the
hypersurface $X_1$ is singular along a subvariety
$\pi_k^{-1}(T)\cap X_k=\widetilde{T}$ for a $T\in {\mathcal T}$, then
over the general point of $T$ the hypersurface $X_k\subset {\mathcal
X}_k$ is given by the equation
$$
a_{T,1}x_1^p-a_{T,0}x_0^p=0,
$$
where, say where, say, $a_{T,1}$ does not vanish on $T$, and
$$
\mathop{\rm ord}\nolimits_T a_{T,0}=l_T\in\{2,\dots,p-1\}.
$$
This means that if $p=2$, then the proof of Proposition~\ref{prop2} is
completed: the variety $X_k$ is a non-singular cyclic cover of
$U$.

Assume now that $p\geqslant 3$. Then along the subvariety
$\pi^{-1}_k(T)\cap X_k$ the variety $X_k$ has a cuspidal
singularity of the type
$$
t^p-s^{l_T}=0.
$$
Taking the normalization of the variety $X_k$ or applying the
obvious sequence of blow ups, we obtain the required variety $U_X$
that covers $U$ cyclically, and complete the proof of Proposition~\ref{prop2} for a cyclic cover of degree $p\geqslant 3$. This completes the
proof of Theorem~\ref{th1}.\qed

\section{Generalizations}

First of all, given that the divisorial
canonicity has been shown in \cite{Pukh06b,EcklPukh2016,Pukh18a}
for many families of Fano complete intersections of index 1 in the
projective space, we obtain the following generalization of
Corollary~\ref{cor3}.

\begin{corollary}\label{cor4}
For a divisorially canonical non-singular
Fano complete intersection $V\subset {\mathbb P}^{M+k}$ of
codimension $k$ and index 1, there are no non-trivial rational
Galois covers $X\stackrel{d:1}{\dashrightarrow} V$ with an abelian
Galois group of order $d\geqslant 2$, where $X$ is a rationally
connected variety; in particular, there are no rational maps
$X\dashrightarrow V$ of degree 2 with $X$ rationally connected.
\end{corollary}

\begin{proof}
We need only to check that the properties ($\star$1) and
($\star$2) are satisfied for a non-singular complete intersection $V$.
The property ($\star$2) is automatic and in order to show ($\star$1), we argue
as in the proof of Corollary~\ref{cor3}: if $p\mathbin{\not|} \mathop{\rm
deg} V$, then a section of $V$ by a general linear
$(k+1)$-subspace in ${\mathbb P}^{M+k}$ works as $N$, and if
$p\mathbin{|} \mathop{\rm deg} V$, we obtain $N$ as the residual
curve,
$$
P\cap V=L+N,
$$
where $L\subset V$ is a general line and $P$ a general linear
$(k+1)$-subspace in ${\mathbb P}^{M+k}$, containing $L$.
\end{proof}

\begin{remark}\label{rk2}
It seems that the claim of Theorem~\ref{th1} is true for
rationally connected Galois rational covers with an arbitrary
Galois group, not necessarily an abelian one. We may assume that
$X$  is a non-singular projective variety and the Galois group of
the extension ${\mathbb C} (V)\subset {\mathbb C}(X)$ is a finite
subgroup of the group $\mathop{\rm Aut} X$. Now the ramification
divisor of the finite cover $U_X\to U$ is invariant with respect
to the action of the Galois group and for that reason is pulled
back from $U$, so that the proof of Theorem~\ref{th1} must work in the
non-abelian case, too.
\end{remark}

\begin{conjecture}\label{conj2}
The claim of Theorem~\ref{th1} holds without the
assumption that the Galois group is abelian.
\end{conjecture}

\begin{remark}\label{rk3}
The fact that there are no rationally connected
rational double covers for a hypersurface $V\subset {\mathbb P}$
of degree $M+1$, where $M\geqslant 5$, was obtained from the
property that every pair $(V,\frac{1}{n}D)$, where $D\sim nH$, is
canonical. However, for certain special hypersurfaces $V$ this is
not true: for instance, if $V\cap T_oV$ is a cone with the vertex
at the point $o$. On the other hand, it seems probable that the
claim of Theorem \ref{th1} is true for any smooth hypersurface of index 1
for $M\geqslant 5$ --- and, possibly, for four-dimensional
quintics and three-dimensional quartics, for which the technique
of \cite{Pukh05} does not prove the property of divisorial
canonicity. Note also that it is hard to see any obstructions for
the Conjecture~\ref{conj1} to be true at least for any smooth hypersurface
of index 1 and dimension $\geqslant 4$. At the same time, as
Segre's example shows, for certain smooth three-dimensional
quartics Conjecture~\ref{conj1} does not hold because they are unirational,
see \cite{IM}.
\end{remark}

\section{Historical remarks and acknowledgements}
As far as I know, the first time when a conjecture on non-unirationality of
certain rationally connected three-folds (in that case, conic
bundles) was explicitly stated, was in \cite{Fano1931} by Gino
Fano himself. In \cite{Kollar1996} Koll\'{a}r suggested an
approach to proving the non-unirationality of primitive Fano
hypersurfaces: if the variety parameterizing rational curves of
some arbitrary fixed degree passing through a general point $o\in
V$ contains no rational curves, then there are no rational
surfaces on $V$ containing this point, which of course implies
non-unirationality. The suggested approach motivated a far
reaching investigation of the space of rational curves on
hypersurfaces in \cite{BeheshtiStarr,BeheshtiKumar}. For the
various explicit results on unirationality (solving the
unirationality problem affirmatively) see, for instance,
\cite{HarrisMazurPand1998,ConteMurre,Conte2001,ConteMarMurre2009,
ConteMarchisio}. A topic that can be linked to the conjecture on
absolute rigidity is endomorphisms, especially rational
endomorphisms of rationally connected varieties; there are quite a
few papers on that subject, see, for instance,
\cite{AproduKebPeternell2008,HwangNakayama2011,AmerikBogomRov2011,
Zhang2012,Zhang2014}.

The author is grateful to the colleagues in the Divisions of
Algebraic Geometry and Algebra at Steklov Institute of Mathematics
for the interest to his work, and to the colleagues-algebraic
geometers at the University of Liverpool for the general support.

Finally, I would like to thank both referees for their work on my
paper and a number of useful suggestions.


\ifx\undefined\bysame
\newcommand{\bysame}{\leavevmode\hbox to3em{\hrulefill}\,}
\fi

\end{document}